\documentclass[12pt]{amsart}
\usepackage{amsmath,amssymb,fullpage}

\title[Rational inner functions on a square-matrix polyball]{Rational inner functions on a square-matrix polyball}

\author[Grinshpan]{Anatolii Grinshpan}
\author[Kaliuzhnyi-Verbovetskyi]{Dmitry~S.~Kaliuzhnyi-Verbovetskyi}
\author[Vinnikov]{Victor Vinnikov}
\author[Woerdeman]{Hugo J.~Woerdeman}
\address{Department of Mathematics \\
Drexel University\\
3141 Chestnut St.\\
Philadelphia, PA, 19104}
\email{\{tolya,dmitryk,hugo\}@math.drexel.edu}
\address{Department of Mathematics\\
Ben-Gurion University of the Negev\\
Beer-Sheva, Israel, 84105} \email{vinnikov@math.bgu.ac.il}

\dedicatory{Dedicated to the blessed memory of Cora Sadosky, our
dear friend and colleague.}
\thanks{AG, DK-V, HW were partially supported by NSF grant DMS-0901628.
DK-V and VV were partially supported by BSF grant 2010432.}


\theoremstyle{plain}

\newtheorem{thm}{Theorem}[section]
\newtheorem{cor}[thm]{Corollary}

\newtheorem{prop}[thm]{Proposition}

\numberwithin{equation}{section}

\newcommand{\beq}{\begin{equation}}
\newcommand{\eeq}{\end{equation}}

\newcommand{\rmat}[3]{\ensuremath{{#1}}^{#2\times #3}}

\theoremstyle{remark}
\newtheorem{rem}[thm]{Remark}

\numberwithin{equation}{section}

\newcommand{\bbm}{\begin{bmatrix}}
\newcommand{\ebm}{\end{bmatrix}}

\begin{document}

\begin{abstract}
 We establish the existence of a finite-dimensional unitary
realization for every matrix-valued rational inner function from
the Schur--Agler class on a unit square-matrix polyball. In the
scalar-valued case, we characterize the denominators of these
functions. We also show that every polynomial with no zeros in the
closed domain is such a denominator. One of our tools is the
Kor\'{a}nyi--Vagi theorem generalizing Rudin's description of
rational inner functions to the case of bounded symmetric domains;
we provide a short elementary proof of this theorem suitable in
our setting.
\end{abstract}

\maketitle

\section{Introduction}\label{sec:Intro}

In this paper, we study rational inner functions on the Cartesian
product of square-matrix Cartan's domains of type I, i.e., on a
unit square-matrix polyball,
\begin{multline*}\label{eq:polyball}
{\mathcal B}=\mathbb{B}^{\ell_1\times
\ell_1}\times\cdots\times\mathbb{B}^{\ell_k\times
\ell_k}\\
=\Big\{Z=(Z^{(1)},\ldots,Z^{(k)})\in\rmat{\mathbb{C}}{\ell_1}{\ell_1}\times\cdots\times
\rmat{\mathbb{C}}{\ell_k}{\ell_k}\colon\|Z^{(r)}\|<1,\
r=1,\ldots,k\Big\}.\end{multline*} We can interpret points of
$\mathcal{B}$ as block-diagonal matrices
$Z=\bigoplus_{r=1}^kZ^{(r)}$ with $\|Z\|<1$. Then $\mathcal{B}$ is
a special case of domain $\mathcal{D}_\mathbf{P}$ defined by a
matrix polynomial $\mathbf{P}$ as a set of
$z=(z_1,\ldots,z_d)\in\mathbb{C}^d$ satisfying
$\|\mathbf{P}(z)\|<1$ (see \cite{AT,BB,GKVVW1}); in this case,
 $\mathbf{P}=Z$  viewed as a polynomial in matrix entries
$z^{(r)}_{ij}$,   $i,j=1,\ldots , \ell_r$, $r=1,\ldots , k,$ and
$d=\sum_{r=1}^k \ell_r^2$. In particular, $\mathcal{B}$ is a
special case of a Cartesian product of (not necessarily square)
matrix Cartan's domains of type I (see \cite{Hua,GKVVW2}). The
distinguished (or Shilov) boundary of ${\mathcal B}$ consists of
$k$-tuples of unitary matrices,
$$ \partial_S {\mathcal B} = \Big\{Z=(Z^{(1)},\ldots,Z^{(k)})\in\rmat{\mathbb{C}}{\ell_1}{\ell_1}\times\cdots\times
\rmat{\mathbb{C}}{\ell_k}{\ell_k}\colon Z^{(r)*}Z^{(r)} = I,\
r=1,\ldots,k\Big\},$$ which can also be interpreted as a set of
block-diagonal unitary matrices. Notice that the unit polydisk
$\mathbb{D}^d$ is a special case of a unit square-matrix polyball
where $k=d$, and $\ell_r=1$ for $r=1,\ldots,k$. We consider matrix
functions in variables $z_{ij}^{(r)}$. We denote the corresponding
$d$-tuple of variables by $z$ and, for a function $F$, we identify
$F(z)=F(Z)$. An $s\times s$ matrix-valued function $F$ is rational
inner if each matrix entry $F_{\alpha\beta}$ is a rational
function in $d$ variables $z_{ij}^{(r)}$ which is regular in
$\mathcal{B}$, and $F$ takes unitary matrix values at each of its
regular points on the distinguished boundary $\partial_S {\mathcal
B}$. Notice that the zero variety of the least common multiple of
the denominators of the rational functions $F_{\alpha\beta}$ in
their coprime fraction representation has an intersection with
$\partial_S {\mathcal B}$ of relative Lebesgue measure zero, which
can be proved using an argument analogous to that of \cite[Lemma
6.3]{BSV}; thus almost all points of $\partial_S {\mathcal B}$ are
regular points of $F$.

Define
\begin{multline*}
 {\mathcal T}_{Z}
= \Big\{ T=\bigoplus_{r=1}^k T^{(r)}\colon \\
T^{(r)}=[T_{ij}^{(r)}]_{ij=1}^{\ell_r },\
 T_{ij}^{(r)}
{\rm \ are \ commuting \ operators \ on \ a \ Hilbert \ space  \
and} \ \| T \| < 1 \Big\}.
\end{multline*}
For $T \in {\mathcal T}_Z$, the Taylor joint spectrum \cite{T1} of
$T$ viewed as a multioperator
$$(T_{ij}^{(r)})_{r=1,\ldots,k;\,i,j=1,\ldots,\ell_r}$$
lies in the domain $\mathcal{B}$, and for a matrix-valued function
$F$ analytic on $\mathcal{B}$ one can define $F(T)$ by means of
Taylor's functional calculus \cite{T2}; see \cite{AT} and a
further discussion in \cite{BB}.  We say that an $s\times s$
matrix-valued function $F$ analytic on $\mathcal{B}$ belongs to
the Schur--Agler class $\mathcal{SA}_Z(\mathbb{C}^s)$ associated
with $\mathcal{B}$, or rather with its defining polynomial
$\mathbf{P}=Z$, if its associated Agler norm,
$$ \| F \|_{{\mathcal A},Z}= \sup_{T
\in {\mathcal T}_Z} \| F(T) \|$$ is at most 1.
 In the scalar case, $s=1$, we
simply write $\mathcal{SA}_Z$. In the case of the unit polydisk
$\mathbb{D}^d$, this class coincides with the classical
Schur--Agler class $\mathcal{SA}_d$ studied in the seminal paper
of Agler \cite{Ag}. Notice that $\mathcal{SA}_Z(\mathbb{C}^s)$ is
a subclass of the Schur class $\mathcal{S}_Z(\mathbb{C}^s)$ of
$s\times s$ matrix-valued contractive analytic functions on
${\mathcal B}$. It follows from \cite{Varo} and \cite{Arv3} that
these classes do not coincide, i.e., the analog of von Neumann's
inequality fails
 when $d\ge 2$, unless ${\mathcal B}=\mathbb{D}^2$. Moreover,
 rational inner functions, which obviously belong to the Schur class, do
 not necessarily belong to the Schur--Agler class: an example of
 a rational inner function on $\mathbb{D}^3$ which is not Schur--Agler was
 given in \cite[Example 5.1]{GKVW}.

In Section \ref{sec:scalar_inner}, we give a characterization of
scalar-valued rational inner functions on $\mathcal{B}$ in terms
of their coprime fraction representation. In Section
\ref{sec:inner_SA}, we describe matrix-valued rational inner
functions on $\mathcal{B}$ that belong to the associated
Schur--Agler class as the functions that have a finite-dimensional
unitary realization. In Section \ref{sec:eventual_Agler_denom}, we
characterize eventual Agler denominators, i.e., those
multivariable polynomials which can be represented as the
denominators of scalar rational inner functions in the
Schur--Agler class $\mathcal{SA}_Z$, in terms of certain
contractive determinantal representations. We also show that every
polynomial with no zeros on the closed domain,
$\overline{\mathcal{B}}$, is an eventual Agler denominator, and we
end with several open questions.

\section{Scalar-valued rational inner functions}\label{sec:scalar_inner}

For a polynomial $p$ in $d$ variables $z_{ij}^{(r)}$,
$i,j=1,\ldots , \ell_r$, $r=1,\ldots , k,$ where $d=\sum_{r=1}^k
\ell_r^2$, we define its reverse with respect to $\mathcal{B}$ as
$$ \overleftarrow{p}(Z)= \prod_{r=1}^k (\det Z^{(r)})^{t_r} \overline{p (Z^{*-1})}, $$ where
$t_r $ is the total degree of $p$ in the variables $z_{ij}^{(r)}$,
$i,j=1,\ldots , \ell_r$. We say that a polynomial $p$ is
$\mathcal{B}$-stable (resp., strongly $\mathcal{B}$-stable) if $p$
does not have zeros in $\mathcal{B}$ (resp., in
$\overline{\mathcal{B}}$).

  The
following result is a generalization of Rudin's characterization
of rational inner functions on the polydisk \cite{Rudin} to the
case of a square-matrix polyball $\mathcal{B}$. It appears in more
generality in \cite[Theorem 3.3]{KoVa}, where Rudin's theorem is
extended to all bounded symmetric domains. We provide a proof that
applies to the specific setting of $\mathcal{B}$, and therefore
requires less machinery.

\begin{thm}\label{rudin} A scalar-valued function $f$ on $\mathcal{B}$ is rational inner
 if and only if there exist a $\mathcal{B}$-stable polynomial $p$ and
nonnegative integers $m_1$, \ldots, $m_k$ such that
\begin{equation}\label{eq:rudin} f(Z) = \prod_{r=1}^k (\det Z^{(r)})^{m_r}
\frac{\overleftarrow{p}(Z)}{p(Z)}. \end{equation} One can choose
$p$ to be coprime with $\overleftarrow{p}$.
\end{thm}

For the proof of Theorem \ref{rudin}, we will need the following
proposition.

\begin{prop}\label{irr} Let $p$ be a $\mathcal{B}$-stable polynomial, and suppose that
$|p(Z)|=1$ for all $Z\in\partial_S {\mathcal B}$. Then there exist
nonnegative integers $m_1$, \ldots, $m_k$ such that $$p(Z) =
\prod_{r=1}^k (\det Z^{(r)})^{m_r}.$$
\end{prop}

\begin{proof} Notice that if $Z^{(r)}$ is unitary for each $r$, then
\begin{equation}\label{prevp} \overleftarrow{p} (Z) p(Z) =  \prod_{r=1}^k (\det Z^{(r)})^{t_r}. \end{equation}
Since $\partial_S {\mathcal B}$ is a uniqueness set for analytic
functions (see, e.g., \cite{Hua}), we have that \eqref{prevp}
holds for all
$Z=(Z^{(1)},\ldots,Z^{(k)})\in\rmat{\mathbb{C}}{\ell_1}{\ell_1}\times\cdots\times
\rmat{\mathbb{C}}{\ell_k}{\ell_k}$. Since $\det Z^{(r)}$ is an
irreducible polynomial in matrix entries $z^{(r)}_{ij}$ (see,
e.g., \cite[Section 61]{Bocher}) we obtain
  that $ p(Z) = \prod_{r=1}^k (\det Z^{(r)})^{m_r}$ for some $m_r \le t_r$, $r=1,\ldots,k$.
\end{proof}

\begin{proof}[Proof of Theorem \ref{rudin}] The sufficiency of the representation \eqref{eq:rudin} for
$f$ to be rational inner is clear. To prove the necessity, we
first write $f=q/p$, with $p$ and $q$ coprime. Since $f$ is
analytic in ${\mathcal B}$, we have that $p$ is ${\mathcal
B}$-stable. Next, for $Z\in\partial_S {\mathcal B}$ we have that
$q(Z)\overline{q(Z^{-1*})} = p(Z)\overline{p(Z^{-1*})}$. Hence the
equality
$$ \prod_{r=1}^k (\det Z^{(r)})^{\tau_r} q(Z)\overline{q(Z^{-1*})} = \prod_{r=1}^k (\det Z^{(r)})^{\tau_r}
 p(Z)\overline{p(Z^{-1*})}$$
 holds for every $Z\in\partial_S\mathcal{B}$,
 where $\tau_r $ is the maximum of the total degrees in $z_{ij}^{(r)}$,
 $i,j=1,\ldots,\ell_r$, in $q$ and $p$.
Then it also holds for all $Z$. Since $q$ and $p$ are coprime, $p$
is a divisor of $ \prod_{r=1}^k (\det Z^{(r)})^{\tau_r}
\overline{q(Z^{-1*})}$. Hence
$$  u(Z) q(Z)= \prod_{r=1}^k (\det Z^{(r)})^{\tau_r} \overline{p(Z^{-1*})}, $$ for some polynomial $u$.
 Observe that $|u(Z)|=1$ on $\partial_S\mathcal{B}$. Then by Proposition \ref{irr} we
 obtain that $u(Z)=\prod_{r=1}^k (\det Z^{(r)})^{\mu_r}$, with some
 $\mu_r\le\tau_r$, and $$q(Z)=\prod_{r=1}^k (\det Z^{(r)})^{\tau_r-\mu_r}
 \overline{p(Z^{-1*})}.$$
 The latter equality implies that $m_r:=\tau_r-\mu_r-t_r\ge 0$,
 where $t_r$ is the total degree of $p$ in variables
 $z^{(r)}_{ij}$, $i,j=1,\ldots,\ell_r$. It follows that
  $$q(Z)=\prod_{r=1}^k (\det Z^{(r)})^{m_r}
 \overleftarrow{p}(Z),$$
 and  \eqref{eq:rudin} holds.
\end{proof}

\begin{rem}\label{rem:sym}\rm Proposition \ref{irr} and Theorem \ref{rudin} also hold when ${\mathcal B}$ is replaced by
a Cartesian product of square-matrix Cartan's domains of type II
\cite{Hua},
$$ {\mathcal B}_{\rm sym} = \Big\{Z=(Z^{(1)},\ldots,Z^{(k)})\in {\mathcal B}\colon
Z^{(r)} = Z^{(r)\top}, r=1,\ldots , k \Big\}, $$ or by a Cartesian
product of mixed type involving Cartan's domains of types I and
II. The proofs are similar after noticing that $\det Z^{(r)}$ as a
polynomial in $z^{(r)}_{ij}$, $i,j=1,\ldots,\ell_r$: $i\le j$, is
also irreducible when $Z^{(r)} = Z^{(r)\top}$; see, e.g.,
\cite[Section 61]{Bocher}.
\end{rem}

\section{Matrix-valued rational inner functions from the Schur--Agler class}
\label{sec:inner_SA}

We now characterize matrix-valued rational inner functions on the
unit square-matrix polyball $\mathcal{B}$, which belong to the
associated Schur--Agler class, in terms of their unitary
realizations. This is a generalization of the result \cite[Theorem
2.1]{BK} for the unit polydisk $\mathbb{D}^d$ which, in turn, is a
matrix-valued extension of the result from \cite{Knese2011} in the
scalar-valued setting (see also an earlier paper \cite{ColeWermer}
for the bidisk case).

\begin{thm}\label{fd} An $s\times s$ matrix-valued function $F$ on $\mathcal{B}$ is rational inner and belongs to the class
$\mathcal{SA}_Z(\mathbb{C}^s)$ if and only if $F$ has a
finite-dimensional unitary realization, i.e., there exist
nonnegative integers $n_1$, \ldots, $n_k$ and a unitary matrix
$$ \begin{bmatrix} A & B \\ C & D \end{bmatrix} \in {\mathbb C}^{(\sum_{r=1}^k \ell_r n_r + s)\times
(\sum_{r=1}^k \ell_r n_r + s)} $$ such that
$$ F(Z) = D + CZ_n (I-AZ_n)^{-1}B, $$
where $Z_n = \bigoplus_{r=1}^k (Z^{(r)} \otimes I_{n_r}). $ If we
write $F=QP^{-1}$, where $P$ and $Q$ are matrix polynomials of
total degree at most $g$ such that $P^*P=Q^*Q$ on
$\partial_S\mathcal{B}$,  then
 $n_1, \ldots, n_k$ can be chosen so that
$n_r \le l_rs\binom{g+d-1}{d}$, $r=1,\ldots,k$.
\end{thm}

For the proof of Theorem \ref{fd}, we will need the following
proposition. Recall that a linear mapping $\Phi \colon {\mathbb
C}^{a\times a} \to  {\mathbb C}^{b\times b}$ is said to be
completely positive if, for every $m\in\mathbb{N}$, the mapping
$\Phi^{(m)}\colon ({\mathbb C}^{a\times a})^{m\times m} \to
({\mathbb C}^{b\times b})^{m\times m}$ defined by
$(\Phi^{(m)}(A))_{ij}=\Phi(A_{ij})$,\ $i,j=1,\ldots,m$, is
positive, that is, it maps every positive semidefinite matrix $A$
to a positive semidefinite matrix $\Phi^{(m)}(A)$.

\begin{prop}[{\cite[Theorem 1]{Choi}}]\label{Choi} Let $\Phi \colon {\mathbb C}^{a\times a} \to  {\mathbb C}^{b\times b}$
 be a completely positive linear mapping. Then there exists $Y\in {\mathbb C}^{a^2b\times b}$ so that
$\Phi(X) = Y^*(X\otimes I_{ab})Y$.
\end{prop}

\begin{proof}[Proof of Theorem \ref{fd}.] The sufficiency part is analogous to that of \cite[Theorem 6.1]{BSV}. To prove the necessity,
let $F=QP^{-1}$, where $P$ and $Q$ are matrix polynomials of total
degree at most $g$, $P^*P=Q^*Q$ on $\partial_S\mathcal B$, and
assume that $F\in\mathcal{SA}_Z(\mathbb{C}^s)$. Then by
\cite[Theorem 1.5]{BB} there exist separable Hilbert spaces
$\mathcal{K}_r$ and analytic functions $H_r$ on $\mathcal{B}$ with
values linear operators from $\mathbb{C}^s$ to
$\mathbb{C}^{\ell_r}\otimes\mathcal{K}_r$ such that
$$P(W)^*P(Z)-Q(W)^*Q(Z)=\sum_{r=1}^kH_r(W)^*\Big( (I-W^{(r)*}Z^{(r)})\otimes I_{{\mathcal K}_r}\Big)H_r(Z),
\quad Z,W\in\mathcal{B}.$$ Letting $Z=W=tU$ where $|t|<1$ and
$U\in\partial_S\mathcal B$, we obtain
\begin{equation}\label{cd}\frac{P(tU)^*P(tU)-Q(tU)^*Q(tU)}{1-|t|^2}=\sum_{r=1}^kH_r^*(tU)H_r(tU).\end{equation}
Since $P(tU)^*P(tU)=Q(tU)^*Q(tU)$ for all
$U=(U^{(1)},\ldots,U^{(r)})\in\partial_S\mathcal B$ and $|t|=1$,
the numerator of the left-hand side of \eqref{cd} is a polynomial
in $t$ and $\overline t$ which vanishes on the variety
$1-t\overline t=0$. Therefore the left-hand side of \eqref{cd} is
a polynomial in $t$ and $\overline t$ and a trigonometric
polynomial in matrix entries $u_{ij}^{(r)}$,
$i,j=1,\ldots,\ell_r$, $r=1,\ldots,k$. Let $P_\alpha$ and
$Q_\alpha$ be the coefficients of $z^\alpha$ in the polynomials
$P$ and $Q$, respectively, and let $H_{r,\alpha}$ be the
coefficient of $z^\alpha$ in the Maclaurin series for $H_r$, where
for $z=(z_1,\ldots,z_d)$ and $\alpha=(\alpha_1,\ldots,\alpha_d)$
we set $z^\alpha=z_1^{\alpha_1}\cdots z_d^{\alpha_d}$. Then the
zeroth Fourier coefficient of the left-hand side of \eqref{cd} as
a trigonometric polynomial in variables $u_{ij}^{(r)}$ is
$$\frac1{1-|t|^2}\sum_{|\alpha|\le g}(P^*_\alpha P_\alpha-Q^*_\alpha Q_\alpha)|t|^{2|\alpha|},$$
where $|\alpha|=\alpha_1+\cdots +\alpha_d$. Note that the
preceding expression is a polynomial in $|t|^2$ of degree at most
$g-1$. The zeroth Fourier coefficient of the right-hand side of
\eqref{cd} (for sufficiently small $t$),
$$\sum_{r=1}^k\sum_{\alpha}H^*_{r,\alpha}H_{r,\alpha}|t|^{2|\alpha|},$$
is therefore a polynomial in $|t|^2$ as well, and
$H^*_{r,\alpha}H_{r,\alpha}=0$ for $|\alpha|\ge g$.

Consider now the completely positive map $\Phi_r\colon {\mathbb
C}^{\ell_r\times\ell_r} \to   {\mathbb C}^{s\binom{g+d-1}{d}
\times s\binom{g+d-1}{d}}$ defined via
$$ \Phi(X) = {\rm col}_{|\alpha |\le g-1}  (H_{r,\alpha}^*)(X\otimes I_{{\mathcal K}_r} ) {\rm row}_{|\alpha |\le g-1}
(H_{r,\alpha}). $$ Then by Proposition \ref{Choi} we can find
matrices $Y_r\in\mathbb{C}^{a^2b\times b}$ such that $\Phi_r (X) =
Y_r^*(X\otimes I_{ab})Y_r$, where $a=\ell_r$ and
$b=s\binom{g+d-1}{d}$. Writing $Y_r ={\rm row}_{|\alpha|\le g-1}
(Y_{r,\alpha})$, we can form a polynomial
$$G_r(Z)=\sum_{|\alpha|\le g-1}Y_{r,\alpha}z^\alpha
$$
with the coefficients in $\mathbb{C}^{\ell_rn_r\times s}$, where
$n_r=\ell_rs\binom{g+d-1}{d}$, so that $$H_r(W)^*\Big(
(I-W^{(r)*}Z^{(r)})\otimes I_{{\mathcal K}_r} \Big)H_r(Z)=G_r(W)^*
\Big((I-W^{(r)*}Z^{(r)})\otimes I_{n_r}\Big) G_r(Z),\quad
r=1,\ldots,k,$$ and
\begin{equation}\label{cdpoly}
P(W)^*P(Z)-Q(W)^*Q(Z)=\sum_{r=1}^kG_r(W)^*
\Big((I-W^{(r)*}Z^{(r)})\otimes I_{n_r}\Big)G_r(Z).
\end{equation}
Rearranging the terms in \eqref{cdpoly}, we obtain
$$P(W)^*P(Z)+\sum_{r=1}^kG_r(W)^*\Big(W^{(r)*}Z^{(r)}\otimes I_{n_r}\Big)G_r(Z)=Q(W)^*Q(Z)+\sum_{r=1}^kG_r(W)^*G_r(Z).$$
Therefore
$$\begin{bmatrix}(Z^{(1)}\otimes I_{n_1}) G_1(Z)\\ \vdots\\ (Z^{(k)}\otimes I_{n_k})G_k(Z)\\ P(Z)\end{bmatrix}h\mapsto \begin{bmatrix}G_1(Z)\\\vdots\\ G_k(Z)\\ Q(Z)\end{bmatrix}h$$
is a linear and isometric map from the span of the elements on the
left to the span of the elements on the right. It may be extended
(if necessary) to a unitary matrix
$\begin{bmatrix}A&B\\C&D\end{bmatrix}$ so that
\begin{align*}
&AZ_nG(Z)+BP(Z)=G(Z)\\
&CZ_nG(Z)+DP(Z)=Q(Z),
\end{align*}
for every $Z$; here $Z_n=\bigoplus_{r=1}^k(Z^{(r)}\otimes
I_{n_r})$. Solving the first equation above for $G(Z)$ and then
plugging the result into the second equation yield
$$F(Z)=Q(Z)P^{-1}(Z)=D+CZ_n(I-AZ_n)^{-1}B.$$
\end{proof}

\begin{rem}\label{rem:I-II-III} An analog of Theorem \ref{fd}, with a similar proof, is also valid for Cartesian products of Cartan's
domains of type II or III \cite{Hua} or for  Cartesian products of
mixed type involving Cartan's domains of types I, II, and III. We
recall here that a Cartan domain of type II (resp., III) is a
(lower-dimensional) subset of a square-matrix Cartan's domain of
type I consisting of symmetric (resp., antisymmetric) matrices.
\end{rem}

\section{Eventual Agler denominators} \label{sec:eventual_Agler_denom}
We will say that a polynomial $v$ in $z^{(r)}_{ij}$,
$i,j=1,\ldots,\ell_r$, $r=1,\ldots,k$, is almost self-reversive
with respect to the square-matrix polyball $\mathcal{B}$ if
$\overleftarrow{v}=\gamma v$, for some scalar $\gamma$ with
$|\gamma |=1$.

We have the following generalization of a result that was
announced in \cite{GKWMTNS} for the case of a unit polydisk.

\begin{thm}\label{Aglerdenrev} Let $p$ be a $\mathcal{B}$-stable polynomial which is coprime with $\overleftarrow{p}$. Then the following are equivalent:
\begin{itemize}
    \item[(i)] $p$ is an eventual Agler denominator, that is, there exist nonnegative integers $s_1, \ldots , s_k$
    such that the rational inner function
    $ \prod_{r=1}^k (\det Z^{(r)})^{s_r}\overleftarrow{p}(Z)/p(Z)$ is in the Schur--Agler class $\mathcal{SA}_Z$.
    \item[(ii)] There exists an almost self-reversive polynomial $v$ of multidegree $(s_1,\ldots , s_k)$ such that $p(Z)v(Z)=\det(I-KZ_n)$ for some nonnegative integers
    $n_r$, $r=1,\ldots , k$, and a contractive matrix $K$, where $Z_n = \bigoplus_{r=1}^k ( Z^{(r)} \otimes I_{n_r})$.
\end{itemize}
\end{thm}

\begin{proof}  (i)$\Rightarrow$(ii) Let $f(Z)=\prod_{r=1}^k (\det Z^{(r)})^{s_r}\overleftarrow{p}(Z)/p(Z)$ be in $\mathcal{SA}_Z$.  By Theorem \ref{fd} there exists a $k$-tuple
 $n=(n_1, \ldots , n_k)$ of nonnegative integers and a unitary matrix
$ \left[\begin{smallmatrix} A & B \\ C & D
\end{smallmatrix}\right]
$ such that $$f(Z)= D + CZ_n (I-AZ_n)^{-1} B. $$
 Observe that
$$ f(Z) = D+CZ_n (I-AZ_n)^{-1}B = \frac{\det \begin{bmatrix} I-AZ_n & B \\ -CZ_n & D
\end{bmatrix}}{\det (I-AZ_n)} $$
and
$$ \det \begin{bmatrix}   I-AZ_n & B \\ -CZ_n & D \end{bmatrix}= \det \begin{bmatrix} A & B \\ C & D
 \end{bmatrix}\det \left( \begin{bmatrix} A^* & 0 \\ B^* & 0 \end{bmatrix} -  \begin{bmatrix} Z_n & 0 \\ 0 & -1
 \end{bmatrix} \right)=
\lambda \det (A^* - Z_n),
$$
where $\lambda = \det \left[\begin{smallmatrix} A & B \\ C & D
\end{smallmatrix}\right]$. Hence \begin{equation}\label{pq}
\prod_{r=1}^k (\det Z^{(r)})^{s_r}\overleftarrow{p}(Z) \det
(I-AZ_n) = \lambda p(Z) \det(A^* - Z_n).
\end{equation} Since $p$ is $\mathcal{B}$-stable and coprime with $\overleftarrow{p}$, the polynomials $p$ and $\prod_{r=1}^k (\det Z^{(r)})^{s_r}\overleftarrow{p}$ do
not have common factors. Therefore $p$ divides $\det (I-AZ_n)$,
i.e., there exists a polynomial $v$ so that
$$ p(Z)v(Z)= \det (I-AZ_n). $$
Then $$ \det Z_n \overline{p(Z^{*-1})}\overline{v(Z^{*-1})}  =
 \det (Z_n - A^*). $$ Using \eqref{pq}, we obtain
$$ \prod_{r=1}^k (\det Z^{(r)})^{s_r} \overleftarrow{p} (Z) p(Z) v(Z) = \lambda \prod_{r=1}^k (\det Z^{(r)})^{n_r-{\rm deg}_r p}p(Z)\overleftarrow{p} (Z) \overline{v(Z^{*-1})} (-1)^{\sum_{r=1}^k \ell_r n_r},$$
where ${\rm deg}_rp$ is the total degree of $p$ in the variables
$z_{ij}^{(r)}$.
 After dividing out we see that
$v$ is almost self-reversive and $s_r=n_r-\deg_rp-\deg_rv$.
Clearly, $K=A$ is a contractive matrix.

(ii)$\Rightarrow$(i) Suppose there exists an almost self-reversive
polynomial $v$ such that $p(Z)v(Z)=\det(I-KZ_n)$ with a
contractive matrix $K$ and a $k$-tuple $n=(n_1,\ldots,n_k)$ of
nonnegative integers. Then by a straightforward modification of
\cite[Theorem 5.2]{GKVW} the rational inner function
$$\frac{\prod_{r=1}^k (\det Z^{(r)})^{n_r} \overline{p(Z^{*-1})}\overline{v(Z^{*-1})}}{p(Z)v(Z)}$$
is Schur--Agler. Since $v$ is almost self-reversive, the rational
inner function
$$\frac{\prod_{r=1}^k (\det Z^{(r)})^{n_r-{\rm deg}_r v} \overline{p(Z^{*-1})}}{p(Z)}$$ is Schur--Agler, i.e., $p$ is an eventual Agler denominator.
\end{proof}

\begin{rem} An analog of Theorem \ref{Aglerdenrev} is valid for a
Cartesian product of Cartan's domains of type II, i.e., for a
domain $\mathcal{B}_{\rm sym}$, or for a Cartesian product of
mixed type involving Cartan's domains of types I and II; see
Remark \ref{rem:sym} and Remark \ref{rem:I-II-III}.
\end{rem}

The following result is a generalization of \cite[Corollary
3.2]{GKVVW2} formulated for the polydisk $\mathbb{D}^d$ to the
case of the square-matrix polyball $\mathcal{B}$.

\begin{thm}\label{thm:strongly-stable} Every strongly $\mathcal{B}$-stable polynomial is an eventual Agler denominator.
\end{thm}
\begin{proof}
By \cite[Theorem 3.1]{GKVVW2} there exist a $k$-tuple of
nonnegative integers $n=(n_1,\ldots,n_k)$ and a strictly
contractive matrix
$K\in\mathbb{C}^{(\sum_{r=1}^k\ell_rn_r)\times(\sum_{r=1}^k\ell_rn_r)}$
such that $p(Z)=\det(I-KZ_n)$, where
$Z_n=\bigoplus_{r=1}^k(Z^{(r)}\otimes I_{n_r})$. Then, similarly
to the last paragraph in the proof of Theorem \ref{Aglerdenrev}
(with $v=1$), one shows that $p$ is an eventual Agler denominator.
\end{proof}

\begin{cor}\label{cor:inner-SA}
Let $f$ be a rational inner function on $\mathcal{B}$ which is
regular on $\partial_S\mathcal{B}$. Then there exist nonnegative
integers $s_1,\ldots,s_k$ such that $\prod_{r=1}^k(\det
Z^{(r)})^{s_r}f\in\mathcal{SA_Z}$.
\end{cor}
\begin{proof}
By Theorem \ref{rudin} there exists a stable polynomial $p$ which
is coprime with $\overleftarrow{p}$ and such that \eqref{eq:rudin}
holds with some nonnegative integers $m_1,\ldots,m_r$. Since $f$
is regular on $\partial_S\mathcal{B}$, the polynomial $p$ is
strongly $\mathcal{B}$-stable. By Theorem
\ref{thm:strongly-stable}, $p$ is an eventual Agler denominator.
Therefore $\prod_{r=1}^k(\det Z^{(r)})^{s_r}f\in\mathcal{SA_Z}$
for some nonnegative integers $s_1,\ldots,s_k$.
\end{proof}

The question on whether the assumption of regularity of $f$ on
$\partial_S\mathcal{B}$ in Corollary \ref{cor:inner-SA} can be
removed is open. Another open question is whether Corollary
\ref{cor:inner-SA} holds for {\em matrix-valued} rational inner
functions. Finally, it is interesting to investigate the analogues
of the results in this paper for the unbounded version of the
domain $\mathcal{B}$, i.e., the Cartesian product of matrix
halfplanes. The Cayley transform over the matrix variables
$Z^{(r)}$, $r=1,\ldots,k$, would allow one to obtain a
finite-dimensional realization formula for rational inner
functions on the product of matrix halfplanes; if, in addition,
the Cayley transform over the values of a function is applied,
then one can obtain the corresponding realization formula for
rational Cayley inner functions over the product of matrix
halfplanes (see \cite{BK} for the case of a polyhalfplane, i.e.,
the product of scalar halfplanes). We would also like to mention
\cite{K} where a subclass of Cayley inner functions on the product
of matrix halfplanes, the Bessmertnyi class, was studied.


\begin{thebibliography}{10}

\bibitem{Ag} J.~Agler.
\newblock On the representation of certain holomorphic functions
defined on a polydisc.
\newblock In  Topics in operator theory: Ernst D. Hellinger Memorial
Volume,  {\em Oper. Theory Adv. Appl.},
 Vol.~48, pp. 47--66,
  Birkh\"auser, Basel, 1990.

\bibitem{AT}
C.-G. Ambrozie and D. A. Timotin. Von Neumann type inequality for
certain domains in
 $\mathbb{C}^n$. {\em Proc. Amer. Math. Soc.} 131 (2003), no. 3, 859--869 (electronic).

\bibitem{Arv3} W. Arveson. Subalgebras of C*-algebras. III. Multivariable operator theory.
  {\em Acta Math.} 181 (1998), no. 2, 159--228.


 \bibitem{BB}
 J. A. Ball and V. Bolotnikov. Realization and interpolation for Schur--Agler-class functions on domains with matrix
 polynomial defining function in $\mathbb{C}^n$. {\em J. Funct. Anal.} 213 (2004), no. 1, 45--87.



\bibitem{BK}  J. A. Ball and D. S. Kaliuzhnyi-Verbovetskyi.
  Rational Cayley inner Herglotz--Agler functions: Positive-kernel decompositions and
   transfer-function realizations. {\em Linear Algebra Appl.}  456 (2014), 138--156.

\bibitem{BSV}
 J. A. Ball, C. Sadosky, and V. Vinnikov. Scattering systems with several evolutions and
multidimensional input/state/output systems. {\em Integral
Equations and Operator Theory} 52 (2005), 323--393.


\bibitem{Bocher} M. Bocher, {\it Introduction to higher algebra}. Dover Publications, Inc., New York 1964.



\bibitem{Choi} M. D. Choi, Completely positive linear maps on complex matrices. {\em Linear Algebra and Appl.} 10 (1975), 285--290.

\bibitem{ColeWermer} B. J. Cole and J. Wermer. Ando's theorem and sums of squares.
Indiana Univ. Math. J. 48 (1999), no. 3, 767--791.


\bibitem{GKVVW2} A. Grinshpan, D. S. Kaliuzhnyi-Verbovetskyi, V. Vinnikov, and H. J. Woerdeman.
Contractive determinantal representations of stable polynomials on
a matrix polyball. arXiv:1503.06161.


\bibitem{GKVVW1} A. Grinshpan, D. S. Kaliuzhnyi-Verbovetskyi, V. Vinnikov, and H. J. Woerdeman.
Matrix-valued Hermitian Positivstellensatz, lurking contractions,
and contractive determinantal representations of stable
polynomials. {\em Oper. Theory: Adv. Appl.}, to appear; also
available in arXiv:1501.05527.


\bibitem{GKVW} A. Grinshpan, D. S. Kaliuzhnyi-Verbovetskyi, and H. J. Woerdeman.
Norm-constrained determinantal representations of multivariable
polynomials. {\em Complex Anal. Oper. Theory} 7 (2013), 635--654.

\bibitem{GKWMTNS} A. Grinshpan, D. S. Kaliuzhnyi-Verbovetskyi, and H. J. Woerdeman. The Schwarz Lemma and
the Schur-Agler Class. Proceedings of the 21st International
Symposium on Mathematical Theory of Networks and Systems,
Groningen, 2014.


\bibitem{Hua}
L. K. Hua. {\it Harmonic analysis of functions of several complex
variables in the classical domains.} Translated from the Russian
by Leo Ebner and Adam Kor\'{a}nyi. American Mathematical Society,
Providence, R.I. 1963 iv+164 pp.

\bibitem{K}
D. S. Kalyuzhnyi-Verbovetzkii. On the Bessmertnyi class of
homogeneous positive holomorphic functions on a product of matrix
halfplanes. In Operator theory, systems theory and scattering
theory: multidimensional generalizations, {\em Oper. Theory Adv.
Appl.}, Vol. 157, pp. 139--164, Birkh\"auser, Basel, 2005.

\bibitem{Knese2011} G. Knese. Rational inner functions in the Schur-Agler class of the polydisk. {\em Publ. Mat.},
55 (2011), 343--357.

\bibitem{KoVa} A. Kor\'anyi and S. Vagi. Rational inner functions on bounded symmetric domains.  {\em Trans. Amer. Math. Soc.} 254 (1979), 179--193.


\bibitem{Rudin}
W. Rudin, {\it Function theory in polydiscs}. W. A. Benjamin,
Inc., New York--Amsterdam 1969.



\bibitem{T1}
J. L. Taylor.  A joint spectrum for several commuting operators.
{\em J. Functional Analysis} 6 (1970), 172--191.

\bibitem{T2}
J. L. Taylor. The analytic-functional calculus for several
commuting operators. {\em Acta Math.} 125 (1970), 1--38.

\bibitem{Varo}
N.~Th. Varopoulos.
\newblock On an inequality of von {N}eumann and an application of the metric
  theory of tensor products to operators theory.
\newblock {\em J. Functional Analysis}, 16:83--100, 1974.




\end{thebibliography}
\end{document}